\documentclass[11pt]{article}

\usepackage[top=0.9in, left=1.0in, bottom=0.9in, right=1.0in]{geometry}

\usepackage{hyperref}  
\usepackage{setspace} 
\usepackage{color}
\usepackage{graphicx}
\usepackage{amsfonts}
\usepackage{amssymb}
\usepackage{amsmath}
\usepackage{amsthm}

\newtheorem{theorem}{Theorem}[section]

\newtheorem{lemma}[theorem]{Lemma}
\newtheorem{conjecture}[theorem]{Conjecture}
\newtheorem{claim}{}[theorem]

\title{Graphs that contain a $K_{1,2,3}$ and no induced subdivision of $K_4$ are $4$-colorable}

\author{Rong Chen\\
\\
Center for Discrete Mathematics,\ \ Fuzhou University\\
Fuzhou,\ \ P. R. China}

\begin{document}

\maketitle

\footnote{

Emails: rongchen@fzu.edu.cn (R. Chen). }


\begin{abstract}
In 2012, L\'ev\^eque, Maffray, and Trotignon conjectured that each  graph $G$ that contains no induced subdivision of $K_4$ is $4$-colorable. In this paper, we prove that this conjecture holds when $G$ contains a $K_{1,2,3}$.

\end{abstract}

\section{Introduction}
All graphs considered in this paper are finite, simple, and undirected. Let $G$ be a graph.
A {\em proper coloring} of $G$ is an assignment of colors to the vertices of $G$ such that no two adjacent vertices receive the same color. A graph is {\em $k$-colorable} if it has a proper coloring using at most $k$ colors. 
A {\em subdivision} of $G$ is a graph obtained from $G$ by replacing the edges of $G$ with independent paths of length at least one between their end vertices. An {\em $\text{ISK}_4$} is a graph that is isomorphic to a subdivision of $K_4$.
For a graph $H$, we say that $G$ {\em contains} $H$ if $H$ is isomorphic to an induced subgraph of $G$, and otherwise, $G$ is {\em $H$-free}. For a family $\mathcal{F}$ of graphs, we say that $G$ is {\em $\mathcal{F}$-free} if $G$ is $F$-free for every graph $F\in\mathcal{F}$.

Given a graph with a large chromatic number, it is natural to ask whether it must contain induced subgraphs with particular properties. A family $\mathcal{F}$ of graphs is said to be {\em $\chi$-bounded} if there exists a function $f$ such that for every graph $G\in \mathcal{F}$, every induced subgraph $H$ of $G$ satisfies $\chi(H)\leq f(\omega(H))$. 
The notion of $\chi$-bounded families was introduced by Gy\'arf\'as in \cite{AG}.
Scott \cite{Scott} proposed the following conjecture: for every graph $H$, the class defined by excluding all subdivision of $H$ as induced subgraphs is $\chi$-bounded, and proved this conjecture in the case that $H$ is a forest. However, this conjecture was disproved by Pawlik et al. \cite{Pawlik}. In 2012, L\'ev\^eque, Maffray, and Trotignon \cite{LMT} showed that the chromatic number of $\text{ISK}_4$-free graphs is bounded by a constant $c$, which implies that Scott's Conjecture is true when $H=K_4$. Since no example of an $\text{ISK}_4$-free graph whose chromatic number is 5 or more is known, L\'ev\^eque, Maffray, and Trotignon \cite{LMT} in 2012 proposed the following conjecture.
\begin{conjecture}\label{conj}(\cite{LMT}, Conjecture 1.7.)
Every $\text{ISK}_4$-free graph is $4$-colorable.
\end{conjecture}

L\'ev\^eque, Maffray, and Trotignon \cite{LMT} proved that every $\{\text{ISK}_4, \text{wheel}\}$-free graph is 3-colorable. Trotignon and Vu\v skovi\'c \cite{TV} showed that every $\text{ISK}_4$-free graph with girth at least five is 3-colorable. In the same paper, they further conjectured that every $\{\text{ISK}_4, \text{triangle}\}$-free graph is 3-colorable, which was proved by Chudnovsky et al. \cite{CLSSTV} in 2019. Le \cite{Le} showed that every $\{\text{ISK}_4, \text{triangle}\}$-free graph is 4-colorable and every $\text{ISK}_4$-free graph is 24-colorable. 
Chen et al. \cite{CCCFL} improved Le's upper bound to 8.  
In this paper, we prove 
\begin{theorem}\label{main-thm}
For every $\text{ISK}_4$-free graph $G$, if $G$ contains a $K_{1,2,3}$, then $G$ is $4$-colorable.
\end{theorem}


\section{Preliminary}
A {\em cycle} is a connected $2$-regular graph. Let $G$ be a graph. 
For any $u,v\in V(G)$, we often use $u\sim v$ to denote $uv\in E(G)$, and $u\nsim v$ to denote $uv\notin E(G)$. For any $U\subseteq V(G)$, let $G[U]$ be the subgraph of $G$ induced on $U$.  For any vertex disjoint subgraphs $H$ and $H'$ of $G$, 
if there are no edges between $H$ and $H'$, we say they are {\em anticomplete}; and if each vertex in $V(H)$ is adjacent to all vertices in $V(H)$, we say they are {\em complete}. Let $N_G(H)$ be the set of vertices in $V(G)-V(H)$ that have a neighbour in $H$. We say that $N_G(H')\cap V(H)$ is the {\em attachment} of $H'$ to $H$.  We say that an induced path $P$ with ends $v_1,v_2$ is a {\em direct connection} linking $H$ and $H'$ if $v_1$ is the only vertex in $V(P)$ having a neighbour in $V(H)$ and $v_2$ is the only vertex in $V(P)$ having a neighbour in $V(H')$. 
For any $x,y,z\in V(G)$, 
let $d_G(x,y)$ denote the length of a shortest $(x,y)$-path in $G$. Let $P$ be an $(x,y)$-path and $Q$ a $(y,z)$-path. When $P$ and $Q$ are internally disjoint paths, 
let $PQ$ denote the $(x,z)$-path $P\cup Q$. Evidently, $PQ$ is a path when $x\neq z$, and $PQ$ is a cycle when $x=z$. When $u,v\in V(P)$, let $P(u,v)$ denote the subpath of $P$ with ends $u, v$. When there is no confusion, all subscripts are omitted.  

Given an induced cycle $C$ of $G$ and a vertex $v\in V(G)-V(C)$, the vertex $v$ is {\em linked to $C$} if there are three induced paths $P_1,P_2,P_3$ such that
\begin{itemize}
\item $V(P^*_1\cup P^*_2\cup P^*_3)$ is disjoint from $V(C)$;
\item each $P_i$ has $v$ as its one end and the other end in $V(C)$, and there are no other edges between $P_i$ and $C$;
\item for $1\leq i<j\leq 3$, we have $V(P_i)\cap V(P_j)=\{v\}$;
\item if $x\in V(P_i)$ is adjacent to $y\in V(P_j)$, then either $v\in\{x,y\}$ or $\{x,y\}\subseteq V(C)$; and
\item if $v$ has a neighbour $c\in V(C)$, then $c\in V(P_i)$ for some $i$.
\end{itemize}

Lemma \ref{easy}  obviously holds, which will be used frequently without reference.
\begin{lemma}(\cite{CLSSTV}, Lemma 9.)\label{easy}
If $G$ is an $\text{ISK}_4$-free graph, then no vertex of $G$ can be linked to an induced cycle of $G$.
\end{lemma}

\section{Proof of Theorem \ref{main-thm}}
To prove Theorem \ref{main-thm}, we need a characterization of $\{\text{ISK}_4, K_{3,3}, K_{2,2,2}, \text{prisms}\}$-free graphs containing a $K_{1,2,3}$. 

\begin{lemma}\label{vertex out H}
Let $G$ be an $\{\text{ISK}_4, K_{3,3} ,K_{2,2,2}\}$-free graph and $H$ be a maximal induced $K_{1,2,n}$ in $G$ with $n\geq2$. Let $v$ be a vertex in $V(G)-V(H)$. Then the attachment of $v$ to $H$ is either empty, or consists of one vertex or of one edge.
\end{lemma}
\begin{proof}
Let $A_1=\{a\}, A_2=\{b_1,b_2\}$ and $A_3=\{c_1,c_2,\ldots,c_n\}$ be the three sides of the tripartition of $H$. Since $G$ has no $\text{ISK}_4$, the following obviously holds.
\begin{claim}\label{in2}
If $v$ has a neighbour in $A_i,A_j$, then $v$ is anti-complete to $A_k$, where  $\{i,j,k \}=\{1,2,3 \}$.
\end{claim}

\begin{claim}\label{in A1}
When $v\sim a$, Lemma \ref{vertex out H} holds.
\end{claim}
\begin{proof}[Subproof.]
Without loss of generality we may assume that $v$ has a neighbour in $A_2\cup A_3$, for otherwise Lemma \ref{vertex out H} holds.
Assume that $v\sim b_1$. Then $v$ is anti-complete to $A_3$ by \ref{in2}. By the maximality of $H$, we have $v\nsim b_2$, so $N(v)\cap V(H)=\{a,b_1\}$. That is, Lemma \ref{vertex out H} holds. Hence, we may assume that $v$ is anti-complete to $A_2$. By symmetry we may assume that $n\geq3$, $c_1\sim v\sim c_2$ and $v\nsim c_3$, for otherwise 
the lemma holds. Then $c_1$ is linked to $b_1c_2b_2c_3b_1$, which is a contradiction.
\end{proof}

By \ref{in A1}, we may assume that $v\nsim a$.
\begin{claim}\label{in A2A3}
If $v$ has a neighbour in $A_2$ and $A_3$, then Lemma \ref{vertex out H} holds.
\end{claim}
\begin{proof}[Subproof.]
Without loss of generality we may assume that $b_1\sim v\sim c_1$. Assume that $v\sim b_2$. Since $v$ is not linked to $b_1c_1b_2c_i$ for each $2\leq i\leq n$, we have $v\sim c_i$. That is, $v$ is complete to $A_2\cup A_3$, so $G$ contains a $K_{2,2,2}$, which is a contradiction. 
Hence, $v\nsim b_2$. Since $v$ is not linked to $b_1c_1b_2c_i$ for each $2\leq i\leq n$, we have $v\nsim c_i$. So $N(v)\cap V(H)=\{b_1,c_1\}$.
\end{proof}

By \ref{in A1} and \ref{in A2A3}, we may assume that $|N(v)\cap V(H)|\geq2$ and $N(v)\cap V(H)\subseteq A_i$ for some $2\leq i\leq 3$. When $A_i=A_2$, since $A_i=N(v)\cap V(H)$, the vertex $c_1$ is linked to $vb_1ab_2v$. Similarly, we can show that $A_i\neq A_3$ This proves Lemma \ref{vertex out H}.
\end{proof}

If a graph has a unique cycle, we say that it is {\em unicyclic}.
A {\em prism} is a graph consisting of three vertex disjoint paths $P_1=x_1\ldots y_1$, $P_2=x_2\ldots y_2$, and $P_3=x_3\ldots y_3$ of length at least one, such that $x_1x_2x_3$ and $y_1y_2y_3$ are triangles and there are no edges between these paths except those of the two triangles.

\begin{lemma}\label{K1,2,3-}
Let $G$ be an $\{\text{ISK}_4, K_{3,3}, K_{2,2,2}, \text{prisms}\}$-free graph and $H$ be a maximal induced subgraph of $G$ that is isomorphic to a $K_{1,2,n}$ for some integer $n\geq3$. Let $A_1,A_2,A_3$ be the three sides of the tripartition of $H$ with $|A_1|=1, |A_2|=2$ and $|A_3|=n$. If $G\neq H$, then the attachment of $C$ to $H$ is either empty, a clique or $A_1\cup A_2$ for any component $C$ of $G-V(H)$.
\end{lemma}
\begin{proof}
Set $A_1:=\{a\}, A_2:=\{b_1,b_2\}$ and $A_3:=\{c_1,c_2,\ldots,c_n\}$. Let $C$ be a component of $G-V(H)$. Assume that the attachment of $C$ to $H$ is neither empty nor a clique. 
\begin{claim}\label{a}
$a$ has a neighbour in $V(C)$. 
\end{claim}
\begin{proof}[Subproof.]
Assume not. Since the attachment of $C$ to $H$ is not a clique, there are $x_1,x_2\in A_i$ having a neighbour in $V(C)$  for some $2\leq i\leq 3$. Let $y_1,y_2\in V(C)$ with $y_1\sim x_1$ and $y_2\sim x_2$. By Lemma \ref{vertex out H}, $y_1\neq y_2$. Let $P$ be a shortest $(y_1,y_2)$-path in $C$. Without loss of generality we may assume that $(x_1,x_2, y_1,y_2)$ is chosen such that $P$ is minimum. When some vertex $y\in V(A_j)$ is anti-complete to $P$, since $a$ has no neighbour in $V(C)$, the vertex $y$ is linked to $ax_1y_1Py_2x_2a$, where $\{i,j\}=\{2,3\}$. So each vertex in $V(A_j)$ has a neighbour in $V(P)$. By the choice of $(x_1,x_2, y_1,y_2)$, we have that $A_i=A_3$, $A_3-\{x_1,x_2\}$ is anti-complete to $P$, and by symmetry we may assume that $\{y_k\}=N(b_k)\cap V(P)$ for each $1\leq k\leq2$. For any $x_3\in A_3-\{x_1,x_2\}$, since $\{a,x_3\}$ and $P$ are anti-complete, $x_3$ is linked to $ab_1y_1Py_2b_2a$, which is a contradiction.
\end{proof}
\begin{claim}\label{N(C)}
$a\in V(H)\cap N(C)\subseteq A_1\cup A_i$ for some integer $2\leq i\leq3$.
\end{claim}
\begin{proof}[Subproof.]
By \ref{a}, we have $a\in V(H)\cap N(C)$. Let $x\in A_2\cup A_3$ have a neighbour in $V(C)$. Let $y_1,y_2\in V(C)$ with $y_1\sim a$ and $y_2\sim x$. Note that $y_1$ maybe equal to $y_2$. Let $P$ be a shortest $(y_1,y_2)$-path in $C$. Without loss of generality we may assume that $x$ is chosen such that $|V(P)|$ is minimum. Assume that $x\in A_i$ for some $2\leq i\leq3$. Set $\{i,j\}=\{2,3\}$. 
Since \ref{N(C)} holds when $A_j$ and $C$ are anti-complete, there is a vertex $x'\in A_j$ with $N(x')\cap V(C)\neq\emptyset$. 
When $x'$ has a neighbour in $V(P)$, by the choice of $x$, we have $\{y_2\}=N(x')\cap V(P)$, so $x'$ is linked to $ay_1Py_2xa$, which is a contradiction. So $x'$ has no neighbour in $V(P)$. Let $Q$ be a direct connection linking $x'$ and $P$ with $Q\subset C$. Then $|V(Q)|\geq1$ as $x'$ has no neighbour in $V(P)$. Let $q$ be the end of $Q$ that has a neighbour in $V(P)$. By the choice of $P$, either $|N(q)\cap V(P)|=1$ or $\{q_1,q_2\}\subseteq N(q)\cap V(P)=V(P(q_1,q_2))$ with $1\leq d_P(q_1,q_2)\leq2$. No matter which case happens, $G[V(P)\cup\{a,x,x'\}]$ contains an $\text{ISK}_4$ or a prism as $|V(Q)|\geq1$, which is a contradiction.
\end{proof}
\begin{claim}\label{N(C)+1}
$|V(H)\cap N(C)|=3$.
\end{claim}
\begin{proof}[Subproof.]
Assume not. Since $V(H)\cap N(C)$ is not a clique, by \ref{N(C)}, $a\in V(H)\cap N(C)\subseteq A_1\cup A_3$ and $|A_3\cap N(C)|\geq3$. 
By symmetry we may assume that $c_1,c_2,c_3\in N(C)$. Let $K$ be a minimal connected induced subgraph of $G$ with  $\{c_1,c_2,c_3\}\subset V(K)$ and $K-\{c_1,c_2,c_3\}\subseteq C$. Since $K$ is unicyclic, a path or a subdivision of $K_{1,3}$ whose degree-1 vertices are in $\{c_1,c_2,c_3\}$, the graph $G[V(K)\cup A_2]$ contains an $\text{ISK}_4$ or a $K_{3,3}$, which is not possible.
\end{proof}

Assume that $V(H)\cap N(C)=\{a,c_1,c_2\}$. Let $P$ be an induced $(c_1,c_2)$-path with interior in $C$. Then $G[V(P)\cup A_2\cup\{c_3\}]$ is an $\text{ISK}_4$, which is not possible. So $V(H)\cap N(C)=A_1\cup A_2$ by \ref{N(C)} and \ref{N(C)+1}. That proves Lemma \ref{K1,2,3-}.
\end{proof}


Let $G$ be a graph. 
For any integer $k\geq0$, a {\em $k$-cutset} in $G$ is a subset $S\subset V(G)$ of size $k$ such that $G-S$ is disconnected. A {\em proper $2$-cutset} of $G$ is a 2-cutset $\{a,b\}$ such that $a\nsim b$, $V(G)-\{a,b\}$ can be partitioned into two non-empty sets $X$ and $Y$ so that $X$ and $Y$ are anticomplete and neither $G[X\cup \{a,b\}]$ nor $G[Y\cup \{a,b\}]$ is an $(a,b)$-path. 
A square $S=\{v_1,v_2,v_3,v_4\}$ is an induced cycle $C$ of length four such that $v_1,v_2,v_3,v_4$ occur on $C$ in this order. A {\em link} of $S$  is an induced path $P$ with ends $p,p'$ such that either $p=p'$ and $N(p)=S$ or $N(p)=\{v_1,v_2\}$ and $N(p')=\{v_3,v_4\}$, and there are no edges between $S$ and the interior vertices of $P$. A {\em rich square} is a graph $K$ that contains a square $S$ such that there are at least two components in $K-S$, each of which is a link of $S$. Evidently, $K_{2,2,2}$ is the smallest rich square.

\begin{lemma}(\cite{LMT}, Lemma 3.3.)\label{K3,3}
Let $G$ be an $\text{ISK}_4$-free graph that contains a $K_{3,3}$.  Then either $G$ has a clique cutset of size at most three, or $G$ is complete bipartite or complete tripartite.
\end{lemma}
\begin{lemma}(\cite{LMT})\label{prism}
Let $G$ be an $\text{ISK}_4$-free graph. If $G$ contains a prism or a rich square, then $G$ is the line graph of a graph with maximum degree $3$ or a rich square, or $G$ has a clique cutset of size at most three, or $G$ has a proper $2$-cutset.
\end{lemma}
\begin{lemma}(\cite{Le})\label{line graph}
If $G$ is the line graph of a graph with maximum degree $3$ or a rich square, then $\chi(G)\leq4$.
\end{lemma}


\begin{proof}[Proof of Theorem \ref{main-thm}.]
Assume that Theorem \ref{main-thm} is not true. Let $G$ be a minimal counterexample to this theorem. Then $\chi(G)=5$ and $G$ has neither a proper 2-cutset nor a clique cutset. By Lemmas \ref{prism} and \ref{line graph}, $G$ is $\{K_{2,2,2}, \text{prisms}\}$-free. By Lemma \ref{K3,3}, $G$ is $K_{3,3}$-free.
Let $H$ be a maximal induced subgraph of $G$ that is isomorphic to a $K_{1,2,n}$ for some integer $n\geq3$. Let $A_1, A_2$ and $A_3$ be the three sides of the tripartition of $H$ with $|A_1|=1, |A_2|=2$ and $|A_3|=n$. Evidently, $G\neq H$ and $\chi(G-A_3)\geq4$, for otherwise $\chi(G)\leq \chi(G-A_3)+1\leq4$. Since $G$ is connected and has no clique cutset, by Lemma \ref{K1,2,3-}, the attachment of $C$ to $H$ is $A_1\cup A_2$ for any component $C$ of $G-V(H)$. Then $\chi(G-A_3)=\chi(G)$ as $\chi(G-A_3)\geq4$, which is a contradiction to the minimality of $G$.
\end{proof}

\section{Acknowledgments}
This research was partially supported by grants from the National Natural Sciences
Foundation of China (No. 11971111).



\end{document}